\documentclass[11pt, oneside]{amsart}

\newcommand{\hide}[1]{}

\usepackage{amssymb}
\usepackage{graphicx}
\usepackage{amsmath}

\usepackage{geometry}

\usepackage[colorlinks=true,urlcolor=blue]{hyperref}

\def\textcolor#1{}

\newcommand{\R}{\mathbb{R}}

\renewcommand{\rho}{\varrho}
\renewcommand{\phi}{\varphi}

\renewcommand{\theta}{\vartheta}

\DeclareMathOperator{\core}{c}

\DeclareMathOperator{\diam}{diam}

\DeclareMathOperator{\Vol}{Vol}

\renewcommand{\geq}{\geqslant}
\renewcommand{\leq}{\leqslant}

\theoremstyle{theorem}
\newtheorem{theorem}{Theorem}[section]
\newtheorem{lemma}[theorem]{Lemma}

\newtheorem*{conjecture}{Conjecture}

\newtheorem{MainTheorem}{Theorem}

\theoremstyle{definition}
\newtheorem{definition}[theorem]{Definition}

\theoremstyle{remark}
\newtheorem*{remark}{\textsc{Remark}}

\newcounter{reminder}

\newtheoremstyle{claim}
  {}
  {}
  {\itshape}
  {0pt}
  {\scshape}
  {.}
  { }
  {\thmname{#1}\thmnumber{ #2}\thmnote{ (#3)}}

\theoremstyle{claim}

\numberwithin{equation}{section}

\title[Sausage body is a solution for a reverse isoperimetric problem]{A sausage body is a unique solution for a reverse isoperimetric problem}
\subjclass[2010]{52A30, 52A38; 53A07, 52A39, 52A40, 52B60}
\author[Roman Chernov]{Roman Chernov}
\author[Kostiantyn Drach]{Kostiantyn Drach}
\author[Kateryna Tatarko]{Kateryna Tatarko}

\date{}

\address{Jacobs University Bremen, Research I, Campus Ring 1, 28759 Bremen, Germany}
\email{roman1chernov@gmail.com}
\email{k.drach@jacobs-university.de, kostya.drach@gmail.com}
\address{University of Alberta, 677 Central Academic Building, Edmonton, AB, T6G 2H1, Canada}
\email{tatarko@ualberta.ca}

\thanks{The research of the first two authors was partially supported by the advanced grant ``HOLOGRAM'' of the European Research Council (ERC), which is gratefully acknowledged. The second author is also grateful to Cornell University for their hospitality during his visit there. We thank Alexander Litvak and Vlad Yaskin for reading some parts of the original manuscript and providing some useful remarks. We are also grateful to the anonymous referee for suggesting simplifications to the original proof. }

\begin{document}

\begin{abstract}
We consider the class of \emph{$\lambda$-concave} bodies in $\mathbb R^{n+1}$; that is, convex bodies with the property that each of their boundary points supports a tangent ball of radius $1/\lambda$ that lies locally (around the boundary point) inside the body. In this class we solve a \emph{reverse} isoperimetric problem: we show that the convex hull of two balls of radius $1/\lambda$ (a \emph{sausage body}) is a unique volume minimizer among all $\lambda$-concave bodies of given surface area. This is in a surprising contrast to the standard isoperimetric problem for which, as it is well-known, the unique maximizer is a ball. We solve the reverse isoperimetric problem by proving a \emph{reverse quermassintegral inequality}, the second main result of this paper. 

\vspace{0.3cm}

\noindent{\it Keywords:} $\lambda$-concavity; $\lambda$-convexity; reverse isoperimetric inequality; quermassintegrals; reverse isodiametric inequality. 
\end{abstract}

\maketitle

\section{Introduction}

The classical \emph{isoperimetric inequality} states that if $K$ is an arbitrary domain in $\mathbb R^{n+1}$ with volume $\Vol_{n+1}(K)$ and surface area $\Vol_n(\partial K)$, then
\begin{equation}
\label{Eq:ClassicalIsop}
\Vol_{n+1}(K) \leqslant \frac{\Vol_n(\partial K)^{\frac{n+1}{n}}}{\left(\omega_{n+1}\right)^{\frac{1}{n}} \left(n+1\right)^{\frac{n+1}{n}} },
\end{equation}
where $\omega_{n+1}$ is the volume of the unit ball in $\mathbb R^{n+1}$. It is known that equality in (\ref{Eq:ClassicalIsop}) holds if and only if $K$ is a ball. In other words, the classical isoperimetric inequality asserts that among all domains of given surface area, the ball has the largest possible volume. 

Inequality (\ref{Eq:ClassicalIsop}) has a long and beautiful history, and has been generalized to a variety of different settings (see, for example, surveys \cite{BZ, Ros}). The distinctive point of almost all of these generalizations is that the extreme object is always a ball, as the most symmetric body. On the other hand, the problem can be looked at from a different point of view: under which conditions can one \emph{minimize} the volume among all domains of a given constraint (such as of a given surface area, etc.)? Questions of such type are known as \emph{reverse isoperimetric problems}, and have been actively studied recently. 

The naive attempt of minimizing volume among \emph{all} sets of a given surface area will clearly lead to a trivial result: the $(n+1)$-dimensional volume is zero for every set with empty interior. Therefore, we must consider a family of sets with additional conditions imposed in order to obtain a well-posed reverse isoperimetric problem. One of the natural conditions is convexity or strict convexity.   

One of the first results on the reverse isoperimetric problem is due to Keith Ball. In his celebrated works \cite{Bal1, Bal2} he showed that for any convex body $K$ in $\R^{n + 1}$ there is an affine transformation $T$ such that the volume of $T(K)$ is no smaller than that of the standard simplex of the same surface area; if the bodies are additionally assumed to be symmetric, then the cube is an extreme object. The equality case in Ball's reverse isoperimetric inequalities was completely settled later by Barthe \cite{Ba}. Observe that for Ball's approach the minimizers are no longer balls.

Another approach towards obtaining a reverse isoperimetric inequality was recently taken in \cite{PZh}, where the authors provided a \emph{lower} bound on the area enclosed by a convex curve $\gamma \subset \mathbb R^2$ in terms of its length and the area of the domain enclosed by the locus of curvature centers of $\gamma$. The authors also showed that equality is attained only for a disk. In this respect, the results in \cite{PZh} do not follow the philosophy of a reverse isoperimetric problem. See also \cite{XXZZ}, but again these results, although called `reverse', do not follow the philosophy of a reverse isoperimetric problem.

A different approach towards reversing the classical isoperimetric inequality is by assuming some \emph{curvature constraints} for the boundary. It was pioneered by Howard and Treibergs \cite{HTr95} who proved a sharp reverse isoperimetric inequality on the Euclidean plane for closed embedded curves whose curvature $k$, in a weak sense, satisfies $|k| \leqslant 1$, and whose lengths are in $[2\pi;14\pi/3)$. In~\cite{Ga12}, Gard extended this result to surfaces of revolution that lie in $\mathbb R^3$ and whose principal curvatures, again in a weak sense, are bounded in absolute value by $1$, and the surface areas are not too big. Note that the mentioned curvature restrictions do not imply convexity. 

At the same time, motivated by the study of strictly convex hypersurfaces in Riemannian spaces  (see, for instance, \cite{BM, Bor02, BorDr13}), Borisenko and Drach in a series of papers \cite{BorDr14, BorDr15_1, Dr14} obtained two-dimensional reverse isoperimetric inequalities for so-called \emph{$\lambda$-convex curves}, i.e.\,curves whose curvature $k$, in a weak sense, satisfies $k \geqslant \lambda > 0$. Recently, these results were generalized in \cite{BorNA} for $\lambda$-convex curves in Alexandrov metric spaces of curvature bounded below. The result of Borisenko completely settles the reverse isoperimetric problem for $\lambda$-convex curves.

$\lambda$-convexity is a notion that can be easily transferred to higher dimensions. A convex body in $\mathbb R^{n+1}$ is $\lambda$-convex if the principal curvatures $(k_i)_{i=1}^n$ of the boundary of the body are uniformly bounded, in a weak sense, by $\lambda$, i.e.\,$k_i \geqslant \lambda > 0$ for all $i \in \{1,\ldots,n\}$ (we refer to \cite{BM, BGR, DrPhD} for various results concerning the geometry of multidimensional $\lambda$-convex bodies). It is worth pointing out that the reverse isoperimetric problem for $\lambda$-convex bodies has a non-trivial solution in any dimension, although for dimensions greater than two it is a hard problem that is still widely open (see Subsection~\ref{Subsec:LambdaConvex}).

In this paper we consider a notion, in a sense, dual to the notion of $\lambda$-convexity. In particular, we consider so-called \emph{$\lambda$-concave bodies} in $\mathbb R^{n+1}$. These are the convex sets such that the principal curvatures of their boundaries satisfy $\lambda \geqslant k_i \geqslant 0$ for all $i \in \{1, \ldots, n\}$ (in viscosity sense, see Definition~\ref{Def:LambdaConcave}). For $\lambda$-concave bodies we completely solve the reverse isoperimetric problem in any dimension. This is the first result on the reverse isoperimetric problem in $\mathbb R^{n+1}$, besides the celebrated results of Ball and their various extensions, where the inequality is not restricted to curves or surfaces. Moreover, our methods allow us to prove the full family of sharp inequalities involving quermassintegrals of a convex body.

\subsection{Further motivation.}

Part of our motivation, besides previously mentioned work on the reverse isoperimetric problem for $\lambda$-convex domains due to Borisenko and Drach, came from  results on so-called \emph{Will's conjecture}.

If $K$ is a planar convex body with inradius $r$, then the inequality
\begin{equation*}
    \Vol_2(K)\leqslant r\Vol_1(\partial K) - r^2\pi
\end{equation*}
is called Bonnesen's inradius inequality. Equality holds for the sausage body, that is, the Minkowski sum of a line segment and a circle with radius $r$. An extension of Bonnesen's inradius inequality to higher dimensions was conjectured by Wills \cite{W} in 1970. He conjectured that  
\begin{equation*}
\Vol_{n+1}(K) \leqslant r \Vol_{n}(\partial K) - n\,r^{n+1}\,\omega_{n+1},
\end{equation*}
for every convex body $K \subset \mathbb R^{n+1}$ with inradius $r$. This conjecture was proven independently by Bokowski \cite{B} and Diskant \cite{D}. Although the same inequality with the circumradius $R$ of $K$ substituting $r$ is not true in dimensions greater than two (see \cite{D,H}), Bokowski and Heil \cite{BH86} showed that for higher dimensions, in fact, the inequality sign is reversed: 
\begin{equation}
\label{Eq:BHIneq}
\Vol_{n+1}(K) \geqslant \frac{2R}{n}\Vol_n(\partial K) - \frac{(n +2) R^{n+1}}{n}\omega_{n+1}.
\end{equation}
In \cite{BH86} inequality \eqref{Eq:BHIneq} was obtained as a corollary of the following more general result
\begin{theorem}[\cite{BH86}]
For an arbitrary convex body $K \subset \R^{n+1}$ with circumradius $R$, the inequalities 
\begin{equation}\label{BHMainIneq}
    c_{ijk}R^iW_i(K) + c_{jki}R^jW_j(K) + c_{kij}R^kW_k(K) \geqslant 0,
\end{equation}
hold for every $0 \leqslant i < j < k \leqslant n+1$, where $c_{pqr} = (r - q)(p + 1)$.\qed
\end{theorem}
Here $W_i(K)$ is the \textit{quermassintegral of order $i$} of the convex body $K$, $i \in \{1, \dots, n+1\}$ (see the next subsection and Section~\ref{Sec:Background} for details). Quermassintegrals can be viewed as geometric quantities assigned to a convex body that are a higher-dimensional generalization of the integral curvature of a closed curve, and can be explicitly calculated in terms of the principal curvatures $k_i$ of $\partial K$, provided $\partial K$ is sufficiently smooth (see (\ref{Eq:W_j})). The quermassintegrals of different order provide a natural embedding of the volume $\Vol_{n+1}(K)$, the surface area $\Vol_{n}(\partial K)$ and the volume of the unit ball $\omega_{n+1}$ into the sequence $(W_i (K))_{i=0}^{n+1}$ for which (up to a constant) these are respectively, the zeroth, the first, and the $(n+1)$-th element. Therefore, \eqref{Eq:BHIneq} is a special case of \eqref{BHMainIneq} with $i=0$, $j=1$ and $k = n+1$.  

The form of the Bokowski--Heil inequality (\ref{Eq:BHIneq}) inspired the statement of our main result, Theorem~\ref{Thm:RQMI}, although we use different techniques for the proof. It appears that, having a natural inclusion of the volume, the surface area and the volume of the unit ball into the sequence of quermassintegrals helps to solve the reverse isoperimetric problem for $\lambda$-concave bodies in $\mathbb R^{n+1}$ for every $n \geqslant 1$.

\subsection{The main results.}

Recall that a \emph{convex body} in the Euclidean space $\mathbb R^{n+1}$ is a compact convex set with a non-empty interior. In this paper balls will be closed sets.

\begin{definition}[$\lambda$-concave body]
\label{Def:LambdaConcave}
For a given $\lambda > 0$, a convex body $K \subset \mathbb R^{n+1}$ is called \emph{$\lambda$-concave} if for every $p \in \partial K$ there exists a ball $B_{1/\lambda, p}$ (called a \emph{supporting ball at $p$}) of radius $1/\lambda$ passing through $p$ in such a way that 
\begin{equation}
\label{Eq:LambdaLocal}
B_{1/\lambda, p} \cap U(p) \subseteq K \cap U(p)
\end{equation}
for some small open neighborhood $U(p) \subset \mathbb R^{n+1}$ of $p$.
\end{definition}

Note that since $K$ is assumed to be convex, if $K$ is $\lambda$-concave, then a supporting ball is unique at every point. As for the nomenclature, compare it to the notion of \textit{$\lambda$-convexity} (see \cite{BGR, BorDr13, DrPhD}), for which inclusion (\ref{Eq:LambdaLocal}) is reversed (see also the discussion in Subsection~\ref{Subsec:LambdaConvex}).

If the boundary $\partial K$ of a convex body $K$ is at least $C^2$-smooth, then $K$ is $\lambda$-concave if and only if the principal curvatures $k_i(p)$ for all $i \in \{1,\ldots, n\}$ are non-negative and uniformly bounded above by $\lambda$, i.e.\,$0 \leqslant k_i(p) \leqslant \lambda$ for every $i$ and $p \in \partial K$. Equivalently, in the smooth setting $\lambda$-concavity can be expressed in terms of uniformly bounded \emph{normal curvature}. Let $p \in \partial K$ be a point, $v \in T_p \partial K$ be a vector, $\nu$ be the inward pointing normal to $\partial K$ at $p$, and $\pi(p,v)$ be the two-dimensional plane through $p$ spanned by $v$ and $\nu$. The \emph{normal curvature} $k_{\text{n}}(p, v)$ of $\partial K \subset \mathbb R^{n+1}$ at the point $p$ in the direction of $v$ is defined as
$$
k_{\text{n}}(p,v) := \kappa(p),
$$
where $\kappa(p)$ is the curvature of the planar curve $\partial K \cap \pi(p,v)$ at the point $p$. Using this notion, a convex body $K$ with smooth boundary is $\lambda$-concave if and only if $0 \leqslant k_{\text{n}}(p,v) \leqslant \lambda$ uniformly over $p$ and $v$. In general, $K$ is $\lambda$-concave if the uniform bound on normal curvatures is satisfied in the \emph{viscosity sense} (see \cite[Definition 2.3]{BGR} for a similar approach).

Recall that for a convex body $K \subset \mathbb R^{n+1}$ the \emph{quermassintegral of order $i$} (denoted by $W_i(K)$ with $i \in \{0, \ldots, n+1\}$) arises as a coefficient in the polynomial expansion 
$$
\Vol_{n+1}(K + t B) = \sum\limits_{i = 0}^{n+1} {n+1 \choose i} W_i(K) t^i
$$
known as the \emph{Steiner formula}; here $B$ is the unit Euclidean ball in $\R^{n+1}$ and `$+$' stands for the Minkowski addition; see Section~\ref{Sec:Background} for details. 

\begin{definition}[$\lambda$-sausage body]
A \emph{$\lambda$-sausage body} in $\mathbb R^{n+1}$ is the convex hull of two balls of radius $1/\lambda$ (see Figure~\ref{Fig:Sausage}). 
\end{definition}

\begin{figure}[t]
\centering
\includegraphics[scale=0.75]{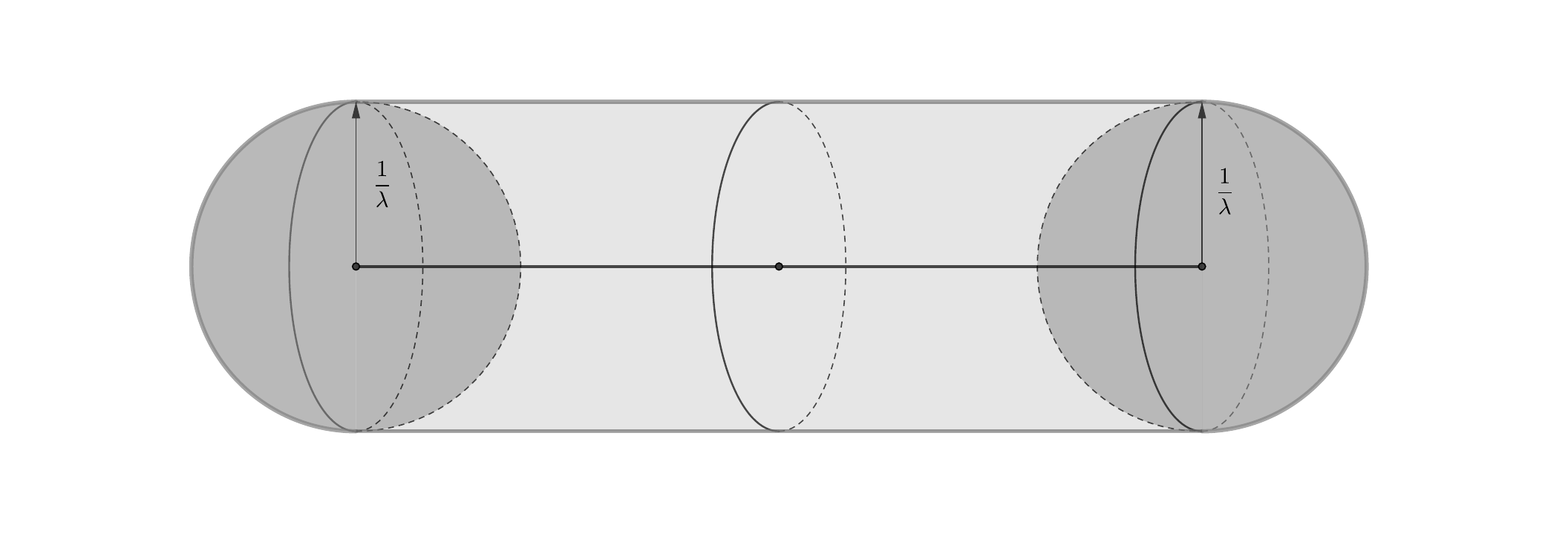}
\caption{A $\lambda$-sausage body.}
\label{Fig:Sausage}
\end{figure}

We are now ready to state the main results of the paper.

\begin{MainTheorem}[Reverse quermassintegral inequality for $\lambda$-concave bodies]
\label{Thm:RQMI}
Let $K \subset \mathbb R^{n+1}$ be a convex body. If $K$ is $\lambda$-concave, then
\begin{equation}
\label{Eq:MainIneq}
(k - j) \frac{W_i(K)}{\lambda^{i}} + (i - k) \frac{W_j(K)}{\lambda^{j}} + (j - i) \frac{W_k(K)}{\lambda^{k}} \geqslant 0
\end{equation}
for every triple $(i, j, k)$ with $0 \leqslant i < j < k \leqslant n+1$. Moreover, equality in (\ref{Eq:MainIneq}) holds if and only if $K$ is a $\lambda$-sausage body. 
\end{MainTheorem}

Since $W_0(K) = \Vol_{n+1}(K)$, $W_1(K) = \Vol_{n} (\partial K) / (n+1)$ and $W_{n+1} (K) = \omega_{n+1}$, inequality (\ref{Eq:MainIneq}) for $i=0$, $j=1$ and $k=n+1$ immediately implies the following result:

\begin{MainTheorem}[Reverse isoperimetric inequality for $\lambda$-concave bodies]
\label{Cor:RIPI}
\label{Thm:RevIsopConcave}
Let $K \subset \mathbb R^{n+1}$ be a convex body. If $K$ is $\lambda$-concave (for some $\lambda>0$), then
\begin{equation}
\label{Eq:MainStatement}
\Vol_{n+1}(K) \geqslant \frac{\Vol_n(\partial K)}{n \lambda} - \frac{\omega_{n+1}}{n \lambda^{n+1}},
\end{equation}
where $\omega_{n+1}$ is the volume of the unit ball in $\mathbb R^{n+1}$. Moreover, equality holds if and only if $K$ is a $\lambda$-sausage body.\qed
\end{MainTheorem}

Theorem~\ref{Thm:RQMI} (and hence Theorem~\ref{Cor:RIPI}) for $n=1$ and $n=2$ was first proven using different techniques in the bachelor thesis of the first author \cite{RomaThesis}. It should be pointed out that Theorem~\ref{Cor:RIPI} for $n=1$ was suggested earlier in \cite{BorDr15_1}; in that paper the authors also prove a similar result on the two-dimensional sphere.

The paper is organized as follows. In Section~\ref{Sec:Background} we recall some necessary background from convex geometry that will be used in the sequel. In Section~\ref{Sec:MainProof} we provide a proof of the central result of the paper (Theorem~\ref{Thm:RQMI}). Finally, Section~\ref{Sec:Remarks} contains some further remarks on the reverse problems; in particular, in Subsection~\ref{Subsec:RevIsodiam} we obtain a so-called \emph{reverse isodiametric inequality} for $\lambda$-concave bodies, and in Subsection~\ref{Subsec:LambdaConvex} we discuss a connection to the dual problem for \emph{$\lambda$-convex} bodies.  


\section{General background on quermassintegrals and convex geometry}
\label{Sec:Background}


In this section we present some background material and auxiliary lemmas towards the proof of the main result.

The \emph{Minkowski addition} of two convex bodies $K$ and $L$ in $\R^{n+1}$ is defined by
\begin{equation*}
K + L := \{ x + y \colon x \in K,\, y \in L \}.
\end{equation*}
One can rewrite the definition in the following form
\begin{equation*}
K + L = \bigcup\limits_{y \in L} (K + y);
\end{equation*}
that is $K + L$ can be viewed as the set that is covered if $K$ undergoes translations by all vectors in $L$.  Since $K$ and $L$ are convex, then $K + L$ is also convex. For a parameter $t \geqslant 0$, the Minkowski sum $K + tB$, where $B$ is the unit ball in $\mathbb R^{n+1}$, is called the \emph{outer parallel body} for $K$.
The \emph{Minkowski difference} of convex bodies $K$ and $L$ is defined by 
\begin{equation*}
K - L := \{ x\in \R^{n + 1} \colon L + x \subset K \}.
\end{equation*}
Similarly to the operation of addition, we can rewrite the definition of Minkowski difference in the form
\begin{equation*}
K - L = \bigcap\limits_{y \in L} (K - y).
\end{equation*}
For a parameter $t \geqslant 0$, the Minkowski difference $K - tB$ is called the \emph{inner parallel body}.

For a convex body $K \subset \mathbb R^{n+1}$, the $(n+1)$-dimensional volume $\Vol_{n+1}(K + tB)$ of its outer parallel body $K + tB$ is a polynomial in $t$, and by the classical \emph{Steiner formula},  
$$
\Vol_{n+1}(K + t B) = \sum\limits_{i = 0}^{n+1} {n+1 \choose i} W_i(K) t^i,
$$
where $W_i(K)$ is the \emph{quermassintegral of order $i$} of the convex body $K$. 

If the boundary $\partial K$ of the body $K$ is $C^{2}$-smooth, and hence the principal curvatures $k_1, \ldots, k_n$ are well-defined everywhere on $\partial K$, then
\begin{equation}
\label{Eq:W_j}
W_j(K) = \frac1{(n+1){n \choose j-1}} \int \limits_{\partial K} \sigma_{j-1} d\textbf{x} \quad \quad \text{ for } 1\leqslant j\leqslant n+1,
\end{equation}
where $\sigma_0 = 1$ and 
$$
\sigma_l = \sum\limits_{1\leqslant i_1<\dots < i_l \leqslant n} k_{i_1}k_{i_2}\dots k_{i_l}
$$
is the $l$-th symmetric function of principal curvatures. 

By convention, $W_0(K)$ is equal to the $(n+1)$-dimensional volume of the body. In particular, $(n+1) W_1(K)$ is the $n$-dimensional volume (surface area) of $\partial K$, and $(n+1) W_{n+1}(K) =: s_n$ is the $n$-dimensional volume (surface area) of the unit sphere $\mathbb S^n$. Recall that $s_n / (n+1) = \omega_{n+1}$, where $\omega_{n+1}$ is the volume of a unit $n$-ball in $\mathbb R^{n+1}$; hence $W_{n+1}(K) = \omega_{n+1}$.

We will need the following generalization of the Steiner formula for inner parallel bodies (see~\cite[p.\,225]{Sch}):
\begin{equation}
\label{Eq:SchnFormulaGen}
W_{q}(K - B) = \sum \limits_{i = 0}^{n+1-q} (-1)^i \binom{n+1-q}{i}W_{q+i}(K)
\end{equation}
for every $0 \leqslant q \leqslant n+1$. In particular, for $q=0$ we have 
\begin{equation}
\label{Eq:SchnFormula0}
\Vol_{n+1}(K-B) = \sum \limits_{i = 0}^{n+1} (-1)^i \binom{n+1}{i}W_i(K).
\end{equation}

For the later purposes we will also adopt the notation $W_{k, j}(K)$ for the quermassintegral of order $j$ of a convex body $K$ lying in $\mathbb R^k$. In particular, such a distinction is needed for the following \emph{Kubota formula}. 

\begin{lemma}[Kubota formula, {\cite[p.\,301]{Sch}}]
\label{Lem:Kubota}
For given $0 < k \leq n$, let $G_{n+1, k+1}$ be the Grassmann manifold of all $(k+1)$-dimensional linear subspaces of $\mathbb R^{n+1}$, and let $d \bold{P}$ be the probability measure on $G_{n+1, k+1}$ which is invariant under the orthogonal group. Then for every convex body $K$ in $\mathbb R^{n+1}$ and for every integer $j$ with $0 \leq j \leq k$,
\begin{equation*}
\int \limits_{G_{n+1, k+1}} W_{k+1, j}\left(K | P\right) d \bold{P} = \frac{\omega_{k+1}}{\omega_{n+1}} W_{n+1, n - k + j}(K),
\end{equation*}  
where $K|P$ is the orthogonal projection of $K$ onto the $(k+1)$-dimensional linear subspace $P \in G_{n+1, k+1}$. \qed
\end{lemma}

The Kubota formula allows to run inductive arguments over the dimension of the space provided that the class of convex bodies in question is closed under orthogonal projections. As we will see now, this is exactly the case for $\lambda$-concave bodies.

The classical result due to Blaschke implies that local condition (\ref{Eq:LambdaLocal}) is in fact global (see \cite{BrStr89, Mil70}, and \cite{Bla56} for the original result of Blaschke).

\begin{theorem}[Blaschke's ball rolling theorem of $\lambda$-concave bodies]
\label{Thm:BlBallRoll}
Let $K \subset \mathbb R^{n+1}$ be a $\lambda$-concave body. Then 
$$
B_{1/\lambda, p} \subseteq K
$$ 
for every point $p \in \partial K$ and every supporting ball $B_{1/\lambda, p}$ at $p$.\qed
\end{theorem}

From Blaschke's ball rolling theorem it follows that the class of convex $\lambda$-concave bodies in $\mathbb R^{n+1}$ is exactly the class of convex bodies $K \subset \mathbb R^{n+1}$ such that $K = K_{\core} + B_{1/\lambda}$ for some convex set $K_{\core}$. This motivates the following definition. 

\begin{definition}[Core of a $\lambda$-concave body]
\label{Def:Core}
A \textit{core} of a $\lambda$-concave body $K$ is the set $K_{\core} := K - B_{1/\lambda}$, where ``$-$'' denotes the Minkowski difference of convex sets, $B_{1/\lambda}$ is a ball of radius $1/\lambda$. 
\end{definition}
(Compare this notion to the notion of a \emph{kernel} of a convex body \cite[p.\,374]{SY}.) 

It is easy to see that $K_{\core}$ is a convex set in $\mathbb R^{n+1}$; however, the core is not necessarily $\lambda$-concave, even more, $K_{\core}$ is not necessarily a convex \emph{body} in $\mathbb R^{n+1}$. Recall that the \emph{affine hull} of a convex set $S \subset \mathbb R^{n+1}$ is the affine subspace of least dimension that contains $S$. We will call the \textit{dimension of the core} (denoted by $\dim K_{\core}$) to be the dimension of the affine hull of $K_{\core}$. Clearly, $K_{\core}$ is a convex body if its dimension is $n+1$. In these terms, a $\lambda$-concave body is a $\lambda$-sausage body if and only if the dimension of its core is at most one, and hence it is either a point or a segment.

Let $P$ be a $(k+1)$-dimensional subspace of $\mathbb R^{n+1}$, and $K$ be a $\lambda$-concave body in $\mathbb R^{n+1}$. Then 
\[
K | P = (K_{\core} + B_{1/\lambda})|P = K_{\core}|P + B_{1/\lambda}|P
\]
by linearity of orthogonal projections. But $B_{1/\lambda}|P$ is a ball of radius $1/\lambda$ in $P$, while $K_{\core}|P$ is some convex set in $P$. Therefore, $K|P$ is a $\lambda$-concave body in $P$ with the core equal to $K_{\core}|P$. These facts prove the following easy, but structurally important, lemma.

\begin{lemma}[Orthogonal projections of $\lambda$-concave bodies]
\label{Lem:Projection}
Orthogonal projections of $\lambda$-concave bodies are $\lambda$-concave. More precisely, if $K$ is a $\lambda$-concave body in $\mathbb R^{n+1}$ and $P$ is a $(k+1)$-dimensional linear subspace of $\mathbb R^{n+1}$, then $K|P$ is a $\lambda$-concave body in $P$; moreover, if $K$ is a $\lambda$-sausage body, then so is $K|P$. \qed
\end{lemma}

\section{Proof of the reverse quermassintegral inequality for $\lambda$-concave bodies (Theorem~\ref{Thm:RQMI})}
\label{Sec:MainProof}

In this section we will prove the main result of this paper --- Theorem~\ref{Thm:RQMI}. Since the left-hand side of (\ref{Eq:MainIneq}) divided by $\lambda$ is scale-invariant, without loss of generality we can assume that $\lambda=1$. 

The following lemma is an important step towards the proof of the result. It allows to drastically simplify further computations. The proof partly follows the ideas in \cite[Theorem 2]{BH86}.  

\begin{lemma}[On three consecutive indices]
\label{Lem:3Consec}
Theorem~\ref{Thm:RQMI} holds true if and only if it holds true for every triple of consecutive indices $(l, l+1, l+2)$ with $0 \leq l \leq n-1$.
\end{lemma}

\begin{proof}
One direction in this lemma is obvious; so suppose Theorem~\ref{Thm:RQMI} holds true for every triple of consecutive indices. For three consecutive indices $(l, l+1, l+2)$ Theorem~\ref{Thm:RQMI} reads as follows:
\begin{equation}
\label{Eq:3Consecutive}
W_{l}(K) -2 W_{l+1}(K) + W_{l +2}(K) \geq 0,
\end{equation}
and equality holds if and only if $K$ is a sausage body. Then for any given triple $(i, j, k)$ with $0 \leq i < j < k \leq n+1$ applying (\ref{Eq:3Consecutive}) repeatedly, we get
\begin{equation}
\label{Eq:Repeatedly}
W_k - W_{k - 1} \geqslant W_{k-1} - W_{k-2} \geqslant \dots \geqslant W_{j} - W_{j - 1} \geqslant \dots \geqslant W_{i+1} - W_{i} 
\end{equation}
(here we simplify our notation by setting $W_i := W_i (K)$). Estimating the sum of the first $k - j$ and the last $j - i$ differences in~(\ref{Eq:Repeatedly}), we get
\begin{align*}
(W_k - W_{k - 1}) +\ldots+(W_{j+1} - W_j) &\geqslant (k-j)(W_j - W_{j - 1}),\\
(W_j - W_{j - 1}) + \ldots + (W_{i+1} - W_i)  &\leqslant (j - i)( W_j - W_{j - 1}).
\end{align*}
Performing cancellation and dividing both inequalities by $k-j$ and $j - i$ respectively, we obtain
$$
\frac{W_k - W_j}{k - j} \geqslant W_j - W_{j-1}  \geqslant \frac{W_j - W_i}{j - i},
$$
and hence $(W_k - W_j)/(k - j) \geqslant (W_j - W_i)/(j - i)$. This is equivalent to (\ref{Eq:MainIneq}); the inequality is proven. If we have equality in (\ref{Eq:MainIneq}), then we must have equality throughout in (\ref{Eq:Repeatedly}). But equality for a triple of consecutive indices yields that $K$ is a sausage body. This concludes the lemma. 
\end{proof}

Our main approach towards the proof of Theorem~\ref{Thm:RQMI} will be by induction on the dimension of the ambient space. The following two lemmas provide necessary steps to run such an induction.

\begin{lemma}[Reverse inequality for $(0,1,2)$]
\label{Lem:FirstTriple}
For every $n \geq 2$, if Theorem~\ref{Thm:RQMI} holds for the triple $(1,2,3)$, then it also holds for the triple $(0,1,2)$.
\end{lemma}
\begin{proof}

This lemma is a consequence of the general Steiner formula for inner parallel bodies. Indeed, by (\ref{Eq:SchnFormulaGen}) for every integer $q$ with $0 \leq q \leq n+1$ we have
\begin{equation}
\label{Eq:GenSteiner2}
W_{q}(K - B_1) = W_q (K_{\core}) = \sum \limits_{i = 0}^{n+1-q} (-1)^i \binom{n+1-q}{i}W_{q+i}(K) \geq 0.
\end{equation}
Therefore
\begin{equation}
\label{Eq:R}
R:=\sum \limits_{q=0}^{n-2} \binom{n-2}{q}W_q(K_{\core}) = \sum \limits_{q=0}^{n-2} \binom{n-2}{q} \sum \limits_{i = 0}^{n+1-q} (-1)^i \binom{n+1-q}{i}W_{q+i} \geq 0.
\end{equation}

Using the simplified notation $W_j = W_j(K)$, we claim that $R = W_0 - 3 W_1 + 3 W_2 - W_3$. This can be easily seen by using the formalism of generating functions. To a linear combination of quermassintegrals $\sum_{i=0}^{n+1} c_i W_i$ we associate the generating function $\sum_{i=0}^{n+1}c_i x^i$. Using such a formalism, the sum in (\ref{Eq:GenSteiner2}) corresponds to the generating function $x^q (1 - x)^{n+1-q}$. Hence, the sum in (\ref{Eq:R}) corresponds to the generating function
\begin{equation*}
\begin{aligned}
\sum_{q=0}^{n-2} \binom{n-2}{q} x^q (1-x)^{n+1-q} &= (1-x)^3 \cdot \sum_{q=0}^{n-2} \binom{n-2}{q} x^q (1-x)^{n-2-q} \\
&= (1-x)^3 \cdot (x + 1-x)^{n-2} = (1-x)^3.
\end{aligned}
\end{equation*}

Therefore, $R = W_0 - 3 W_1 + 3 W_2 - W_3$, as was claimed. But then, since $R \geq 0$, and $W_1 - 2W_2 + W_3 \geq 0$ by the hypothesis of the lemma, it follows that
\begin{equation}
\label{Eq:(0,1,2,3)}
W_0 - 2W_1 + W_2 = R + W_1 - 2W_2 + W_3 \geq 0.
\end{equation}
This is inequality~(\ref{Eq:MainIneq}) for the triple $(0,1,2)$. In order to conclude equality case, assume that $W_0 - 2W_1 + W_2 = 0$. Inequality (\ref{Eq:(0,1,2,3)}) then implies that $W_1 - 2W_2 + W_3 = 0$ because $R$ is non-negative (by (\ref{Eq:R})). By hypothesis, Theorem~~\ref{Thm:RQMI} holds for the triple $(1,2,3)$, and thus equality case for this triple implies that $K$ is a sausage body. The lemma follows.   
\end{proof}

Recall that the extended notation $W_{k+1, l}(K)$ stands for the quermassintegral of order $l$ of a convex body $K$ in $\mathbb R^{k+1}$. For $0 \leq l \leq k-1$, put
\[
E_{k+1,l} (K) := W_{k+1,\, l}(K) -2 W_{k+1,\, l+1}(K) + W_{k+1,\, l+2}(K).
\]
The next lemma guarantees that this quantity is always non-negative in dimension $n+1$ provided that $l > 0$ and that Theorem~\ref{Thm:RQMI} holds in all lower dimensions.

\begin{lemma}[Reverse inequality for $(l,l+1,l+2)$ with $l \geq 1$]
\label{Lem:NotFirstTriple}
For a given $n \geq 1$, if Theorem~\ref{Thm:RQMI} holds in $\mathbb R^{k+1}$ for every $k$ with $1 \leq k < n$, then Theorem~\ref{Thm:RQMI} holds in $\mathbb R^{n+1}$ for every triple of consecutive indices $(l, l+1, l+2)$ with $1 \leq l \leq n-1$. 
\end{lemma}

\begin{proof}
Let $K$ be a 1-concave body in $\R^{n+1}$, and $l$ be an integer satisfying $1 \leq l \leq n-1$. By the Kubota formula (Lemma~\ref{Lem:Kubota}),
\begin{equation}
\label{Eq:EKubota}
E_{n+1, l}(K) = \frac{\omega_{n+1}}{\omega_{n- l + 1}} \int \limits_{G_{n+1, n-l + 1}} E_{n - l + 1, 0}(K|P) d \bold{P}. 
\end{equation}
From the bounds on $l$ it follows that $2 \leq n-l+1 < n+1$, and hence the Grassmann manifold $G_{n+1, n-l + 1}$ is not the trivial one point set $\{ \mathbb R^{n+1} \}$.

By Lemma~\ref{Lem:Projection}, the set $K|P$ is a $\lambda$-concave body in the $(n-l+1)$-dimensional subspace $P$. By hypothesis of the lemma, Theorem~\ref{Thm:RQMI} holds true for such spaces. Therefore, 
\begin{equation}
\label{Eq:LowDimIneq}
E_{n - l + 1, 0}(K|P) \geq 0 \text{ for every } P \in G_{n+1, n-l+1}, 
\end{equation}
and for any given $P$ equality holds if and only if $K|P$ is a sausage body. Combining (\ref{Eq:EKubota}) and (\ref{Eq:LowDimIneq}) we conclude that $E_{n+1, l}(K) \geq 0$, which is exactly the inequality part in Theorem~\ref{Thm:RQMI} for the triple $(l, l+1, l+2)$ and all $1$-concave bodies in $\mathbb R^{n+1}$.

Let us analyze the equality part of Theorem~\ref{Thm:RQMI} for the triple $(l, l+1, l+2)$. Suppose for a given $1$-concave body $K \subset \mathbb R^{n+1}$ one has $E_{n+1, l}(K) = 0$. Then $E_{n-l+1, 0}(K|P) = 0$ for almost all $P \in G_{n+1, n-l+1}$. By hypothesis, the latter equality implies that $K|P$ is a sausage body, again for almost all $P \in G_{n+1, n-l+1}$. Moreover, since 
\[
K|P = K_{\core} | P + B_1 | P,
\]
and $K|P$ is a sausage body, we conclude that $\dim (K_{\core} | P) \leq 1$ for almost all $P$ in the Grassmannian $G_{n+1, n-l+1}$. Taking into account that the dimension of each $P$ is at least $2$, this gives us that $\dim(K_{\core}) \leq 1$. Therefore, $K$ is a $1$-sausage body. The lemma is proven.
\end{proof}

\begin{proof}[Proof of Theorem~\ref{Thm:RQMI}]
Due to Lemma~\ref{Lem:3Consec}, it is enough to prove the result for every triple of consecutive indices. The claim of Theorem~\ref{Thm:RQMI} is now a consequence of Lemmas~\ref{Lem:FirstTriple} and \ref{Lem:NotFirstTriple} by induction on the dimension of the ambient space, which is done as follows. 

The case $n=1$ forms the base of the induction. In $\mathbb R^2$ the result is obvious, and is just a restatement of Steiner formula (\ref{Eq:SchnFormula0}). Indeed, if $K$ is a $1$-concave body in $\mathbb R^2$, then 
\[
0 \leq \Vol_2(K - B_1) = \Vol_2(K_{\core}) = W_0(K) - 2W_1 (K) + W_2(K)
\]
by (\ref{Eq:SchnFormula0}); this proves the inequality (the triple $(0,1,2)$ is the only possible in dimension two). In order to conclude the inequality part, observe that $\Vol_2(K_{\core}) = 0$ implies $\dim K_{\core} \leq 1$, and hence $K$ is a $1$-sausage body. 

The inductive step is a combination of Lemmas~\ref{Lem:FirstTriple} and \ref{Lem:NotFirstTriple}.   
\end{proof}

\section{Concluding remarks}
\label{Sec:Remarks}

\subsection{The reverse isodiametric inequality.} 
\label{Subsec:RevIsodiam}
In this subsection we want to extend our philosophy of a reverse isoperimetric problem to a so-called \emph{isodiametric inequality}. Recall that a \emph{diameter} of a convex body $K \subset \mathbb R^{n+1}$, denoted as $\diam(K)$, is the following quantity:
$$
\diam(K) = \max \limits_{p,q \in K} |pq|.
$$
In other words, the diameter is the length of the largest segment that connects two points in $K$. The classical isodiametric inequality for convex bodies in $\mathbb R^{n+1}$ asserts that for a given diameter $D$ the ball of radius $D/2$ has the largest volume among all convex bodies of diameter $D$ (see \cite[p.\,383]{Sch}).

One simple observation allows us to prove the \emph{reverse isodiametric inequality}.

\begin{theorem}[Reverse isodiametric inequality for $\lambda$-concave bodies]
\label{Thm:RevIsodiam}
Let $K \subset \mathbb R^{n+1}$ be a convex body. Suppose $K$ is $\lambda$-concave, and let $S_\lambda \subset \mathbb R^{n+1}$ be the $\lambda$-sausage body with
$$
\diam(K) = \diam(S_\lambda).
$$ 
Then
\begin{equation}
\label{Eq:RevIsodiam}
W_i (K) \geqslant W_i(S_\lambda) 
\end{equation}
for every $i \in \{0,1,\ldots,n\}$. Moreover, equality holds if and only if $K$ is a $\lambda$-sausage body.
\end{theorem}
\begin{proof}

Let $p$ and $q$ be a pair of points in $K$ realizing the diameter of $K$. It is easy to see that necessarily $p, q \in \partial K$, and moreover, both tangent planes to $\partial K$ at $p$ and $q$ are perpendicular to the segment $pq$. Therefore, if $B_{1/\lambda, p}$ and $B_{1/\lambda, q}$ are the supporting balls at $p$ resp.\ $q$ of radius $1/\lambda$, then $B_{1/\lambda,p} \subseteq K$ and $B_{1/\lambda,q} \subseteq K$ (by Blaschke's ball rolling theorem (Theorem~\ref{Thm:BlBallRoll})) and the convex hull of $B_{1/\lambda,p} \cup B_{1/\lambda,q}$ is the $\lambda$-sausage body $S_\lambda$ of diameter $|pq| = \diam(K)$. Inequality (\ref{Eq:RevIsodiam}) and the equality case then follow by monotonicity of quermassintegrals with respect to inclusion (see \cite[p.\,282]{Sch}).
\end{proof}

\begin{remark}
Theorem~\ref{Thm:RevIsodiam} implies the following sharp estimate on the $i$-th quermassintegral $W_i = W_i(K)$ of a $1$-concave body $K$ in terms of its diameter $D=\diam(K)$:
$$
W_i \geqslant \omega_{n+1}  + \frac{n-i+1}{n+1} \left(D - 2\right) \omega_n \quad \text{ for every }i \in \{0, \ldots, n\}.
$$
The estimate follows by a direct computation of $W_i(S_1)$.
\end{remark}

\subsection{The reverse isoperimetric problem for $\lambda$-convex domains.}
\label{Subsec:LambdaConvex}

We conclude with a surprising difference between the reverse isoperimetric problems for $\lambda$-convex and $\lambda$-concave bodies. For simplicity we restrict ourselves to the Euclidean space, although everything written below makes perfect sense for constant curvature spaces and even general Riemannian manifolds (with appropriate adjustments).

Recall that a convex body $K \subset \mathbb R^{n+1}$ is \emph{$\lambda$-convex} if for every $p \in \partial K$ there exists a ball $B_{1/\lambda, p}$ of radius $1/\lambda$ with the boundary sphere passing through $p$ in such a way that 
\begin{equation}
\label{Eq:LambdaConvexLocal}
B_{1/\lambda, p} \cap U(p) \supseteq K \cap U(p)
\end{equation}
for some small open neighborhood $U(p) \subset \mathbb R^{n+1}$ of $p$ (see \cite{BorDr13, DrPhD}). 

Although $\lambda$-convexity and $\lambda$-concavity seem to be two notions dual to each other, methods and difficulties in solving the reverse isoperimetric problem in each of these classes are quite distinct. In our paper we completely solved the reverse isoperimetric problem for $\lambda$-concave bodies in $\mathbb R^{n+1}$. 
At the same time, only partial results are currently available for $\lambda$-convex bodies.
In particular, the two-dimensional case of the reverse isoperimetric problem for $\lambda$-convex curves, as we already mentioned in the introduction, is completely solved, see \cite{BorNA, BorDr14, BorDr15_1, Dr14}. For higher dimensions the following conjecture is due to Alexander Borisenko (private communication; see also \cite[Subsection 4.7]{DrPhD}).

\begin{conjecture}[Reverse isoperimetric inequality for $\lambda$-convex bodies]
A $\lambda$-convex lens in $\mathbb R^{n+1}$, that is an intersection of two balls of radius $1/\lambda$, is the unique body that minimizes the volume among all $\lambda$-convex bodies of given surface area.
\end{conjecture}

\begin{remark}
A similar conjecture can be stated for all model spaces of constant curvature. In this case the balls are substituted with convex bodies whose boundary is of constant normal curvature equal to $\lambda$. 
\end{remark}

Apart from the case $n=1$, so far this conjecture was verified only in $\mathbb R^3$ for $\lambda$-convex surfaces \emph{of revolution}, see the research announcement in \cite{Dr3}, and \cite{Dr4}. 

Finally, it is interesting to point out numerous results concerning so-called \emph{ball-polyhedra} (see, for example, the paper of Bezdek et al.\,\cite{BLNP} and references therein). A ball-polyhedron is the intersection of finitely many balls of the same radius. Therefore, this is a dual notion to a $\lambda$-concave polytope. In our terminology we would call them $\lambda$-convex polytopes, and a $\lambda$-convex lens is one of them. Following the ideas of Bezdek et al.\,, Fodor, Kurusa and V\'igh \cite{FKV} introduce a notion of $r$-hyperconvexity, which is $1/r$-convexity in the sense of the definition above. In the same paper the authors prove that a two-dimensional $\lambda$-convex lens is a solution of the reverse isoperimetric problem for $\lambda$-convex curves in $\mathbb R^2$ \cite[Theorem 1.3]{FKV}, which was proven earlier in a sharper version in \cite{BorDr14}. Besides, Fodor, Kurusa and V\'igh \cite{FKV} state a conjecture (attributed to Bezdek) which in our language asserts that the intersection of all balls of radius $1/\lambda$ containing a pair of given points (a \emph{$\lambda$-convex spindle}) is a unique body with smallest volume among all $\lambda$-convex bodies of given surface area. This conjecture is false, at least in $\mathbb R^3$, as the results in \cite{Dr3, Dr4} indicate (it is also not hard to check by a direct comparison of volumes of the conjectural solutions).


\end{document}